\documentclass{article}

\usepackage{graphics} 
\usepackage{epsfig} 
\usepackage{graphicx}
\usepackage{epstopdf}
\usepackage{subfigure}
\usepackage{amsfonts}
\usepackage{hyperref}
\usepackage{color}
\usepackage{tikz}

\usepackage{amssymb,amsmath,amsthm}
\usepackage{epsfig,longtable,setspace,color,xcolor}
\usepackage{mathtools}
\usepackage{lipsum}
\usepackage{amsfonts}

\def\RR{\mathbb R}
\def\NN{\mathbb N}
\def\vf{{\varphi}}
\def\a{{\alpha}}

\newtheorem{Theorem}{Theorem}[section]
\newtheorem{Lemma}{Lemma}[section]

 \title{Multiscale continuum-velocity kinetic model for  vehicular traffic with local and mean field interactions}

\author{J. Calvo$^a$, J. Nieto$^a$, M. Zagour$^b$}

\date{$^a$ Departamento de  Matem\'atica Aplicada and Excellence Research Unit ``Modeling Nature'',  Universidad de Granada, 18071 Granada, Spain

$^b$ Ecole Sup\'erieure de Technologie d'Essaouira, Universit\'e Cadi Ayyad,  BP. 383, Essaouira, Maroc

 \textit{To the memory of Abdelghani  Bellouquid}}

\begin{document}

\maketitle

\begin{abstract}
This paper deals with the modeling and mathematical analysis of vehicular traffic phenomena according to a kinetic theory  approach, where the microscopic state of vehicles is described by: (i) position, 
(ii) velocity, as a continuous variable, and also (iii) activity, namely a variable suited to model the quality of the driver-vehicle micro-system. 
Interactions at the microscopic scale are modeled by methods of game theory with continuous velocity variable, thus leading to the derivation of mathematical models 
within the framework of the kinetic theory of active particles. Short-range interactions and mean field interactions are modeled to depict velocity changes related to passing and clustering phenomena. Finally, we perform an analysis of a constructive method to solve the model.
\end{abstract}

\section{Introduction}\label{INTRO}

The mathematical approach to vehicular traffic modeling can be developed at three different observation and representation scales,
namely the microscopic, mesoscopic and macroscopic scales. Different mathematical structures correspond to
each type of representation:

\begin{itemize}

\item  microscopic scale: ordinary differential equations for the variable representing the  state of each vehicle, viewed as an individual entity;

\item  mesoscopic scale: integro-differential equations for a probability distribution  over the microscopic state of vehicles;

\item macroscopic scale: partial differential equations for locally averaged quantities (moments of the aforementioned probability distribution); tipically density, momentum and energy are considered.

\end{itemize}

The critical analysis proposed in the survey paper~\cite{[BD11]} confronts us with the fact that none of the aforesaid scale approaches is totally satisfactory; a multiscale approach is necessary to obtain a detailed description of the
complex dynamics of vehicles on road. Let us comment now on some references in the literature that complement this point of view.  We quote here the survey on the physics and modeling multi-particle systems~\cite{[HEL01]}. In addition, the critical paper by Daganzo~\cite{[DAG95]} has motivated a large body of research in this field. This contribution points out some drawbacks of the driver-vehicle micro-system, where interactions can even modify the behavior of the driver -whose ability is conditioned by the local flow conditions. This paper~\cite{[DAG95]} has generated various discussions and controversies documented in several papers (of which we highlight \cite{[HJ09]})and also reactions to account for the aforementioned criticisms~\cite{BBNS14,[BDF12]}. The approach of~\cite{[BDF12]} has been further developed by various papers, take for instance the proposal~\cite{[FT15]} using a discrete space variable and implementing the model over networks. {Further technical developments can be found in~\cite{[PSTV15],[PSTV17]}}.

Our paper specifically refers to~\cite{[BDF12]}, where a kinetic model has been proposed  with the following main features:

\begin{enumerate}

\item The approach is developed at the mesoscopic scale to account for the heterogeneous behavior of the driver-vehicle micro-system;

\item The velocity variable is assumed to be  discrete to overcome the difficulty that the number of  micro-systems might not be large enough to ensure continuity of the probability distributions over the microstates;

\item The quality of the road-environment conditions is modeled by a parameter that has an influence on the dynamics of interactions. Such parameter takes values in the interval $[0,1]$, where the extremes of the interval correspond to worst and best conditions respectively.

    \end{enumerate}

  Our paper is based on the kinetic theory for active particles~\cite{[BKS13]} and starts from the achievements obtained in~\cite{[BDF12]}, which, in our view, are quite interesting. We mention here the ability of\cite{[BDF12]} to reproduce the fundamental diagrams -namely mean velocity and flow versus local density- as well as clustering phenomena of vehicles with closed speeds. Our paper aims at  providing further  developments of interest for the applications. In more detail, the following modeling topics are treated:  {(i) Interactions between vehicles accounting on perceived (rather than real ones) quantities of the vehicle flow; these interactions can be both local (short-ranged) and long-ranged, (ii) role of variable road conditions, (iii) dynamics under external actions, such as presence of tollgates. In addition, we consider a continuous velocity distribution rather than discrete velocities.} 
 These new modeling features introduced in this paper constitute what we think is a deep revision of the proposal in~\cite{[BDF12]}.

More precisely, the contents of this paper are as follows: Section \ref{sec2} derives a new mathematical structure, suitable to include the aforementioned features, in addition to those already included in~\cite{[BDF12]}; Section \ref{sec3} shows how specific models can be derived, by inserting into the aforesaid structure models of interactions obtained by a detailed phenomenological interpretation of physical reality. Finally, Section \ref{sec4} provides a brief sketch on how to develop a well-posedness theory for the models introduced in Section \ref{sec3}. 

\section{Mathematical structures}\label{sec2}
 According to the kinetic theory of active particles, models are derived in two steps: The first step consists in proposing a mathematical structure able to capture the most important features of the system under consideration. Then the second step consists in deriving specific models of vehicular traffic by inserting into the aforesaid structure models for the various interactions at the microscopic scale.

 This section develops an approach to the first step by deriving a new, general structure, appropriate to include the specific features defined in Section \ref{INTRO}. The overall content is presented through a sequence of subsections going from the representation of the system to the derivation of the structure. This structure is innovative with respect to the existing literature~\cite{[BDF12]}, as it includes modeling local and long-range interactions, as well as interactions with external actions. This structure will offer the conceptual basis for the derivation of specific models.

\subsection{Non dimensional representation}
Let us consider a one-dimensional flow of vehicles along a road of length $\ell$; the road is assumed to have a single lane. 
Position and velocity variables are denoted by
$x$ and $v$. We introduce 
  $v_\ell$, a limit velocity such that  no vehicle,  simply for
mechanical reasons, can exceed, even in the most favorable environmental conditions. We also introduce $t_c$, the time 
 needed by the fastest vehicle to move along the whole length of the road, that is, $t_c := \ell / v_\ell$. This enables us to write down  dimensionless position, velocity and time variables by means of
 $$
 \bar x:= \frac{x}{\ell},\quad  \bar v:= \frac{v}{v_\ell}, \quad \bar t:=\frac{t}{t_c}\:.
 $$
 We shall drop the overlines in what follows. 
 
Dimensionless variables are used also for macroscopic bulk quantities. For instance, the local number density $\rho= \rho(t,x)$ is obtained by dividing the actual density by $\rho_M$, which is the maximum density of vehicles, corresponding to bumper-to-bumper traffic jam.

The analysis developed in what follows is based on the assumption that the state of the driver-vehicle subsystem is defined, at the microscopic scale, by the variables $(x,v,u) \in [0,1]^3$. According to the kinetic theory of active particles~\cite{[BKS13]}, $u$ is a variable that accounts for the quality of the micro-system. More precisely, $u=0$ corresponds to the worst quality, namely, motion is prevented, while $u=1$ corresponds to the best quality, i.e. an experienced driver operating in a high quality vehicle.

According to~\cite{[BKS13]}, the  driver-vehicle subsystem is an \textit{active particle}, while the internal variable is heterogeneously distributed over the active particles. In addition, the quality of the road -including environmental conditions- is accounted for by a parameter $\a \in [0,1]$, such that $\a=0$ corresponds to the worst quality that prevents motion, while $\a=1$ corresponds to the best conditions. In general $\a$ can depend on space $\a = \a(x)$ to account for the presence of curves, local restrictions, etc.

The overall state of the system is described by the distribution function over the states at the microscopic scale:
\[
f = f (t, x, v, u): \quad \mbox{\bfseries R}_+ \times [0,1] \times [0,1] \times [0,1] \, \to \, \mbox{\bfseries R}_+ \, .
\]
Here as usual we require 
$f$ to be locally integrable, so that $f(t,x,v,u) \, dx \, dv\,du$ denotes the dimensionless distribution of vehicles which at time $t$ have position $x$, velocity $v$ and activity $u$. In particular, the local density 
 is given by
\begin{equation}
\rho (t, x)=\int_0^1 \int_0^1 f(t,x,v,u)\, dv\,du.
\end{equation}
Recall that $\rho$ is normalized in term of $\rho_M$, which in turn imposes a normalization on $f$.
\subsection{Interaction domains and perceived quantities}
The car-driver subsystem, namely the active particle, has a visibility zone 
 $\Omega_v = \Omega_v(x) = [x, x + \ell_v]$, where $\ell_v$ is the visibility length on front of the vehicle. This visibility length is assumed to be proportional to  the quality of the environment $\a$ and much smaller than $\ell$. In addition, the active particle has a sensitivity zone, $\Omega_\ell= [x, x + \ell_s]$, necessary to perceive the flow conditions in $\Omega_\ell$.  In general $\Omega_\ell \subseteq \Omega_v$. However, the opposite case has to be taken into account as well, whenever local conditions of the road prevent visibility. 
 In the sequel calculations will be developed assuming that the visibility zone includes the sensitivity zone. In general,  $\Omega_\ell$ can depend on $f$, which induces an additional nonlinearity.

The driver develops its driving strategy by taking into account his perception of the state of the other vehicles in $\Omega_\ell$. Among those vehicles, the ones that are closer to the driver have a stronger influence on his strategy. To account for this, we introduce  a (much) smaller domain, say $\Omega_s$ within which active particles are supposed to perceive an approximate estimate of the local gradients. 
 Hence (see \cite{BBNS14}) active particles perceive a density $\rho_p[\rho]$ which is higher than the real one in the presence of positive gradients and lower than the real one in the presence of negative gradients. We define \textit{short range interactions} as those involving vehicles in $\Omega_s$ and \textit{long range interactions} as those involving vehicles in $\Omega_\ell \backslash \Omega_s$.

 The approach of the kinetic theory for active particles is such that interactions are modeled by evolutionary stochastic games. Three types of particles (driver-vehicle)  are involved: (i) \textit{candidate} particles  with the micro-state $\{x ,v_*,u_*\}$, (ii) \textit{field} particles (vehicles) with the state $\{x^*,v^*,u^*\}$, and (iii) the \textit{test} particle which is representative of the whole system.  Candidate particles are localized in $x$ and can acquire, with a certain probability, the state of the test particle after interactions with the field particles localized in $\Omega_s$ for short range interactions and in  $\Omega_\ell$ for long range (mean field) interactions.

 The rationale toward modeling proposed in the following is based on the assumption that the activity variable of both candidate and test particles  is not modified by interactions.

\subsection{Mean field interactions}
 The test vehicle is subject to an action of the vehicles in its sensitivity zone $\Omega_\ell$ which can induce a consensus toward a common velocity (e.g. the mean speed within the visibility domain) as well as a  clustering effect. The test vehicle is sensitive to these actions whenever the distance between its speed and the common velocity is below a certain critical threshold.

In general, mean field interactions can be modeled  by an individual-based \textit{acceleration term} $\vf(x,x^*, v,v^*,u,u^*)$ that is applied to the \textit{test} vehicle $(x,v,u)$ by all the vehicles $(x^*,v^*,u^*)$ in its sensitivity domain $\Omega_\ell$. Therefore, the overall acceleration is obtained by integration corresponding to the action of all vehicles in $\Omega_\ell$. Hence:
\begin{equation}
\label{F}
\mathcal{F}[f](t,x,v,u) = \int_{\Lambda}\vf(x,x^*, v,v^*,u,u^*) f(t,x^*,v^*,u^*)\,dx^*\, dv^*\, du^*,
\end{equation}
where $\Lambda:= \Omega_\ell \times [0,1]\times [0,1]$.

\subsection{Short range interactions}
\label{sri}

Short range interactions occur, as mentioned, in a smaller domain $\Omega_s$, sufficient for a candidate particle to perceive a density $\rho_p$ modulated by its gradient. Moreover, it is assumed that the probability distribution of field particles can be approximated by the probability distribution in $x$. 
The description of these interactions requires the modeling of two additional quantities:

 $\bullet$ \textit{The encounter rate} ${\eta[f]}$, which models the number of interactions per unit time between candidate and test particles  with field particles.

   $\bullet$  \textit{The transition probability density} $\mathcal{A}[f](v_*\! \to\! v,v^*,u)$, which defines the probability density that a candidate particle with velocity  $v_*$ falls into the state of the test particle with velocity $v$  after interacting with the field particle with velocity $v^*$.

The actual modeling of short range interactions is based on the assumption that these quantities depend not only on the microscopic state of the interacting particles, but also on the whole distribution function $f$. This dependence, which is highlighted by the square brackets, induces a nonlinearity at the microscopic scale. This nonlinearity can involve both $f$ and its gradients, as we will see in the more general model.  In addition, these interaction terms are allowed to depend on the quality of the road modeled by $\alpha \in [0,1]$. Finally, as a probability density, $\mathcal{A}$ is required to satisfy the condition:
\begin{equation}
\mathcal{A}[f; \alpha]\geq 0, \quad \int_{[0,1]}\mathcal{A}[f; \a](v_* \to v,v^*,u) dv = 1,
\label{tabPro}
\end{equation}
for all possible inputs $v_*$, $v^*$, $u$.

\subsection{Interactions with external actions}
 The test vehicle can be subject to external actions which control its velocity. As an example, tollgates indicate  the maximal speed when the vehicle approaches to the tollgate. Similarly, the exit from the tollgate indicates how the speed can increase to the standard values. The effects caused by speed limits can be also described as an external action.

 The simplest way to model this term consists in using a relaxation-type term:
 \begin{equation}
\label{T}
\mathcal{T}[f](t,x,v,u) = {\mu[f,x]}\big(f_e(x, v_e(x)) - f(t, x, v,u)\big),
\end{equation}
where $\mu [f,x]$ models the intensity of the action, which increases with $\rho$, while the given external density $f_e(x, v_e(x))$ is suppose to adapt the velocity of the vehicles to $v_e(x)$. Note that $\mu$ vanishes at those regions where external actions play no role.

\subsection{A mathematical structure toward modeling}

 This subsection shows how all action models that have been described above can be inserted into a proper mathematical structure deemed to offer the conceptual basis for the derivation of different models. This structure is obtained by a particle balance within a elementary volume in the space of microscopic states, which includes position, velocity (namely the variables of the phase space) and the activity variable. This particle balance includes the free transport term, the transitions due to long-range interactions, the dynamics of short-range interaction (consisting  in a ``gain''  and a ``loss'' term), and the trend to the speed enforced by the external actions.  The resulting structure can be written, at a formal level, as follows:
\[
\label{structure-0}
\partial_t f + v\, \partial_x f +  F [f] = J[f] + \mathcal{T}[f],
\]
where  $F$, $J$, and $\mathcal{T}$ correspond, respectively, to mean field interactions, short-range interactions, and interactions with external actions. More concretely, the structure can be written as
\begin{eqnarray}
\label{structure}
 \partial_t f(t,x,v,u) &+& v \partial_x f(t,x,v,u) + \partial_v\Big (\mathcal{F}[f](t,x,v,u)f(t,x,v,u)\Big) \nonumber \\
& = & \int_{[0,1]^3}{\eta[f]}\mathcal{A}[f; \a](v_* \to v,v^*,u)\,f(t,x,v_*,u)\, f(t, x, v^*, u^*)dv_*\,dv^*\, du^* \nonumber \\
&& - f(t, x, v, u)\!\int_{[0,1]^2} \!\!{\eta[f]}\, f(t, x, v^*, u^*)dv^*\, du^*  +{\mu[f,x]}\big(f_e(x, v_e(x)) - f\big).
\end{eqnarray}
\vskip1cm

\subsection{Critical analysis}

The mathematical structure proposed here includes those features of the complex system under consideration which, according to the authors' opinion, appear to be the most important aspects of the dynamics:  heterogeneity of the driver-vehicle subsystem, aggregation dynamics for vehicles with similar velocities, passing probabilities, variable properties of the road-environment where the dynamics occur and the role of external actions.

The structure can operate as a general framework for the derivation of models, which can be obtained by inserting models of interaction at the microscopic scale into it. These models can be obtained by a phenomenological interpretation of empirical data. The most important reference to this aim is the book by Kerner~\cite{[KER04]} -see also~\cite{[KER08]}-,  which provides an interesting variety of empirical data valid in uniform flow conditions as well as in transient conditions. The main difficulty here is that empirical data are available in steady flow conditions, while individual behaviors in unsteady conditions are quite different from  those in steady conditions. However, a sharp interpretation of data can hopefully lead to models that can be validated by the information delivered by empirical data.

\section{From mathematical structures to models} \label{sec3}
This section develops an approach to the derivation of  specific models of vehicular traffic by inserting into the structure (\ref{structure}) models of interactions at microscopic scale. This  objective is pursued by looking at the modeling of the interaction terms that characterize such structure, namely $\varphi$, $\eta$, $\mathcal{A}$, $\mu$ and $f_e$, such that a good agreement with empirical data, concerning both the fundamental diagram and the emerging behaviors in unsteady flow conditions, is provided.

\subsection{Modeling accelerations}
The acceleration term $\vf$ introduced in (\ref{F}) accounts for mean field interactions, where the test vehicle is subject to an action of the vehicles in its sensitivity zone $\Omega_\ell$ in front of it, which can induce a consensus toward a common velocity. We pose the following phenomenological interpretation of reality: {the test vehicle is sensitive to these actions provided that the distance between its speed $v$ and the observed velocity $v^*$ is below a certain critical threshold $d_c$; this acceleration decays with the distance between the test and field vehicle, taking the sign of $v^* - v$ and being proportional to the quality of both the driver-vehicle  and the road, through  $u$ and $\alpha$}.
Then, we propose the following acceleration term:
\begin{equation}\label{vphi1}
\vf(x,x^*,v,v^*,u):= 
\left\{
\begin{array}{lcl}
\displaystyle \alpha \, u\,  \frac{x^*- x}{\ell_v}  (v^* -v) & if &  |v^* - v| \leq d_c, \\
0 & if &  |v^* - v| > d_c  .
\end{array}\right.
\end{equation}
Note that our definition of $\varphi$ does not depend on $u^*$. Whenever $x^*\in \Omega_{\ell }=[x, x+\ell_v]$ we have $0\leq \frac{x^*-x}{\ell_v}\leq 1$ and taking $z^*=\frac{x^*-x}{\ell_v}$ in (\ref{F}),  
it becomes
\begin{eqnarray}
\label{acceleration2}
\mathcal{F}[f] &=& \ell_v\int_{[0,1]^3} \vf(x,x+\ell_v z^*,v,v^*,u) f(t,x+\ell_v z^*,v^*,u^*)dz^*\,dv^*\, du^* \nonumber\\
&=& \alpha \, u\, \ell_v \,\Big (\mathcal{F}_1[f](t,x) - v\ \mathcal{F}_2[f](t,x)\Big),
\end{eqnarray}
where 
\begin{eqnarray*}
\mathcal{F}_1[f](t,x) &=& \int_0^1 \int_{B^*} \int_0^1 z^* v^* f(t,\ell_v z^*+x,v^*,u^*)dz^*\,dv^*\, du^*,\\
\mathcal{F}_2[f](t,x) &=& \int_0^1 \int_{B^*} \int_0^1 z^* f(t,\ell_v z^*+x,v^*,u^*)dz^*\,dv^*\, du^*,
\end{eqnarray*}
with $B^*=\{v^*\in [0,1]: |v^*-v|<d_c  \}$. In particular,  the derivative of $\mathcal{F}$ with respect to $v$ yields
\begin{equation}
\label{accelerationII}
\partial_v(\mathcal{F}[f])(t,x,v,u) =  - \alpha\, u \, \ell_v \, \mathcal{F}_2[f](t,x).
\end{equation}

\subsection{Modeling the  perceived density}
\label{sec31}
The concept of perceived density was introduced in~\cite{[DEA99]}, where it was suggested that this quantity is greater (smaller) than the real one
whenever positive (negative) density gradients appear. Actually, the following expression can be found in \cite{BBNS14}:
\begin{equation}
\label{pd}
\rho_p[\rho]  =  \rho + \frac{\partial_x \rho}{\sqrt{1 + (\partial_x \rho)^2}}\,\big[(1- \rho)\, H(\partial_x \rho) + \rho \, H(- \partial_x \rho)\big],
\end{equation}
where  $H(\cdot)$ is the Heaviside  function. Thus,  the perceived density increases the value of $\rho$ to the largest  admissible value $\rho = 1$ in the presence of positive gradients, while in the presence of negative gradients the perceived density decreases from the actual density $\rho$ to the lowest admissible value $\rho = 0$.

In our framework, which includes mean field interactions through a force term, this type of models for the perceived density can lead to incorrect. More concretely, in previous models with discrete velocity and no acceleration, the density $\rho$ and the perceived density $\rho_p$ remain between zero and one (in a non dimensional representations), and this fact is essential to define the transition probability densities that are involved in the other terms. However, in our new framework with continuous velocity variable we cannot ensure {\em a priori} that the density $\rho$ will remain below one for all time; this is due to the presence of the new acceleration term.

Here we present a simplified scenario in which the above issue is fixed, keeping the spirit of formula (\ref{pd}) as much as possible.  For that aim we introduce a controlled  perceived density -which, by the way, does not involve gradients, thus producing a simpler mathematical  model. If we assume that the influence of  positive gradients is most relevant when: (i) the actual density is high, which tends to increase the value of the perceived density, (ii) the actual density is low, which tends to decrease the value of the perceived density, then we can propose a piecewise linear version of (\ref{pd}), see Figure \ref{fig0} below.
  Actually, we propose a perceived density ${{\rho_p}}[\rho]$ as a Lipschitz function of $\rho$ (dashed line on the picture) 
\begin{figure}[h]
\begin{center}
\begin{tikzpicture}[=>latex,thick,domain=-1:3,xscale=1.5,yscale=1,scale=1.5]
\draw[<-]  (1.6,0) -- (-0.2,0) ; 
\coordinate [label=right:{$\rho$}] (rho) at (1.6,0.1) ; 
\draw[<-] (0,1.3) -- (0,-0.2) ; 
\draw[-] (0,0) -- (1,1) ; 
\draw[dashed, color=black,line width= 1.5 pt, domain=0:0.2] plot (\x,{\x/2});
\draw[dashed, color=black,line width= 1.5 pt, domain=0.2:0.8] plot (\x,{4*(\x-1/8)/3});
\draw[dashed, color=black,line width= 1.5 pt, domain=0.8:1] plot (\x,{(1+\x)/2});
\draw[dashed, color=black,line width= 1.5 pt] (1.5,1) -- (1,1) ; 
\coordinate [label=left:{${{\rho_p}}[\rho]$}] (das) at (0.8,1) ; 
\coordinate [label=left:{$1$}] (uno) at (-0.05,1) ; 
\draw[-] (uno) -- (0.05,1) ; 
\coordinate [label=below:{$1$}] (dos) at (1,-0.05) ; 
\draw[-] (dos) -- (1,0.05) ; 
\end{tikzpicture}
\caption{Cartoon showing how to construct perceived density maps (dashed line) starting from the identity map.}
\label{fig0}
\end{center}
\end{figure}
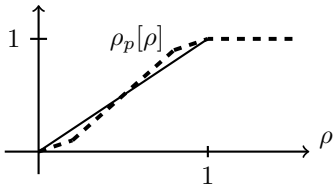
which represents the effect of negative gradients for $\rho \sim 0$,  the effect of  positive gradients for $\rho \sim 1$, and enforces the maximum perceived density to lie between zero and one. The last property will be essential in what follows to construct the transition probability densities modeling short-range interactions. 

\subsection{Modeling the encounter rate}
\label{ss:enc}

The encounter rate $\eta[f]$ introduced in section \ref{sri} refers the rate of interactions between candidate and test particles  with field particles. One can assume that this term grows with the local perceived density, starting from a minimal value corresponding to driving in vacuum conditions $\eta_0$.  We propose the following expression:
\begin{equation}
\eta[f]= \eta_0 \, (1 + \gamma_\eta \, {{\rho_p}}[\rho]),
\end{equation}
where $\gamma_\eta$ is the growth coefficient  and ${{\rho_p}}$ is the perceived density defined in the previous paragraph.

\subsection{Modeling short range interactions}

Let us now consider the \textit{transition probability density}  $\mathcal{A}[f](v_* \to v,v^*,u)$, which models short-range interactions. As said before, this term defines the probability density for a candidate particle with the state $\{x, v_*, u\}$ to fall into the state of the test particle $\{x, v, u\}$ after interaction with the field particles with the state $\{x, v^*, u^*\}$. These notations indicate that the activity variable is not modified by the interaction which, however, modifies the speed.

The modeling approach proposed in our paper is based on the following assumptions, which actually represent, from a continuous point of view, the velocity-discrete table of games introduced in~\cite{[BDF12]}:

 \begin{enumerate}
	\item Short range interactions do not modify the activity, but only the speed.

\item A priori, $\mathcal{A}$ should   depend on the velocities of the interacting particles, on the perceived density, on the activity, and on the quality of the road, i.e. $\mathcal{A}= \mathcal{A}(v_*\to v,v^*,u, \a, {{\rho_p}})$. Here the dynamics is enhanced by the product $\a\, u$, while it is limited by the  perceived density- actually, it is totally prevented whenever ${{\rho_p}} =1$.
\item We assume that the candidate particle with velocity $v_*$, after interacting with the field particles  with velocity $v^*$,  reaches a  new velocity $v$ that belongs to the interval $[v_m, v_M]$  given by 
\begin{equation}
v_m=\max\{0,\min\{v_*, v^*\}-\kappa(1-\alpha u){{\rho_p}}(|v_*-v^*|+\exp(-|v_*-v^*|))\},
\end{equation}
\begin{equation}
v_M=\min\{1,\max\{v_*, v^*\}+\kappa\alpha u(1-{{\rho_p}})(|v_*-v^*|+\exp(-|v_*-v^*|))\}\:.
\end{equation}
In the above $\kappa 
$ is a small constant; this enables $v_m$ to be close to the $\min\{v_*, v^*\}$ and $v_M$ to be close to the $\max\{v_*, v^*\}$. Note that these choices of $v_m$ and $v_M$ guarantee the fulfillment of the condition 
$0< v_m <v_M<1$ even if $v_*=v^*$.
\end{enumerate}

To proceed, we must consider the following two scenarios separately:

\subsubsection{Interaction with faster particles}  If $v_*\leq v^*$, the candidate particle has a trend to increase its speed. The probability of this event taking place decreases  with $v - v_m$ where $v \in [v_m, v_M]$. We propose the following probability density, 

\[
\mathcal{A}[f]=(1-\alpha u(1-{{\rho_p}}))\frac{e^{-\frac{|v-v_*|^2}{\sigma_1}}}{Z(\sigma_1)}+ \alpha u(1-{{\rho_p}})\frac{(v_M-v)^2}{Z_2}, \qquad  v \in [v_m,v_M],
\]
and zero otherwise. This generalizes the table of games defined in the paper~\cite{[BDF12]} in the case of discrete velocity and the activity variables. Here the variance $\sigma_1$ represents a tendency of the particle to modify its velocity under good conditions, and it is given by $\sigma_1=\kappa\alpha u (1-{{\rho_p}})$ with $\kappa<<1$. Moreover, $Z(\sigma)= \int_{v_m}^{v_M}e^{-\frac{|v-v_*|^2}{\sigma}}dv$ and $Z_p= \int_{v_m}^{v_M}(v_M-v)^pdv$ are the respective partition functions, which enforce the normalization condition $\int_0^1 {\cal A}\,  dv =1$.
We note that this probability density has analogous  nonlinear behavior as of the table of games in~\cite{[BDF12]}. More precisely, in good road conditions $\alpha$ and activity $u$, the candidate particle has a tendency to accelerate and reach new velocities greater than its pre-interaction velocity $v_*$. Quite the contrary, decreasing the value of $\alpha$ or $u$ decreases the probability to accelerate (see Figure~\ref{Fig1}-left).

\subsubsection{Interaction with slower particles}  When $v_*> v^*$, the candidate particle has a tendency to decrease its speed. In order to reproduce the behavior of the discrete table of games by~\cite{[BDF12]}, we propose the following dichotomy function as the corresponding probability density ruling these interactions:
    \[
  \mathcal{A}[f]=\alpha u(1-{{\rho_p}})\frac{e^{-\frac{|v-v_*|^2}{\sigma_2}}}{Z(\sigma_2)}+ (1-\alpha u(1-{{\rho_p}}))\frac{v_M-v}{Z_1}. \qquad  v \in [v_m,v_M]
  \]
  Note that $\mathcal{A}[f]$ is set to zero for $v \notin [v_m,v_M]$. 
 Here we have $\sigma_2=\kappa(1-\alpha u (1-{{\rho_p}}))$; the partition functions are defined as in the previous paragraph. We remark that if $\alpha\to 0$ or $u\to 0$, the probability to decelerate is greater than the probability to maintain the velocity. On the contrary, if $\alpha\to 1$ and $u\to 1$, the candidate particle has tendency to maintain it velocity (see Figure \ref{Fig1}-right).
\begin{figure}[h!]
\centering
	 \includegraphics[width=7.5 cm,height=3.5cm]{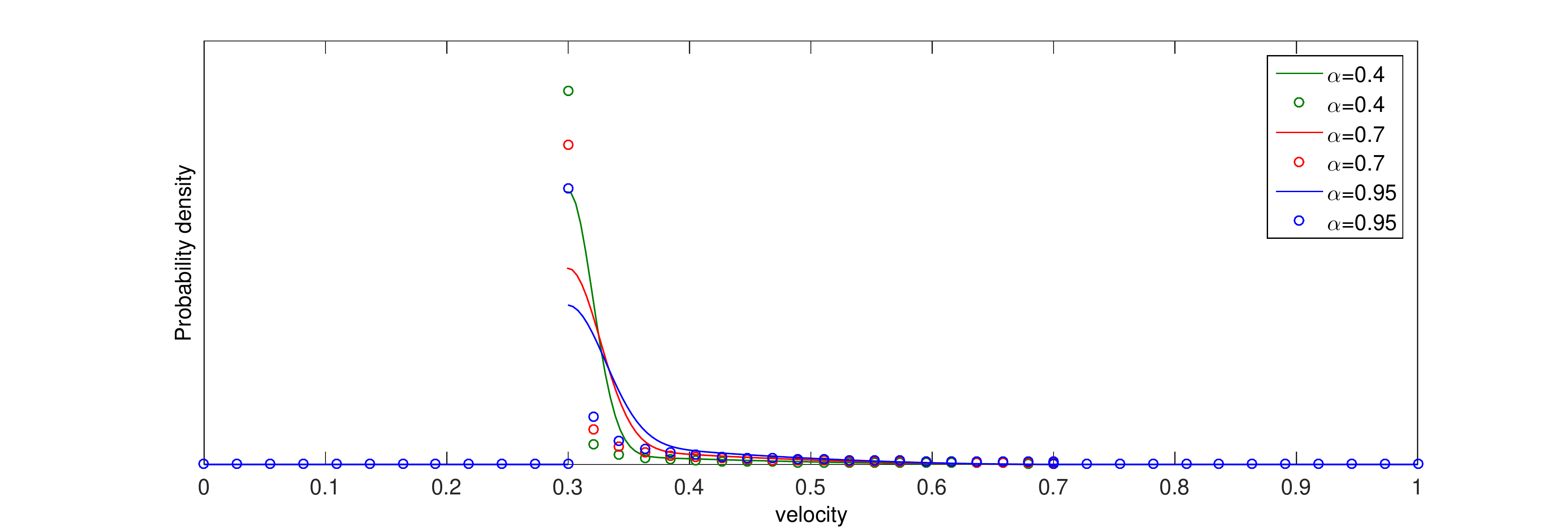}
	 \includegraphics[width=7.5 cm,height=3.5cm]{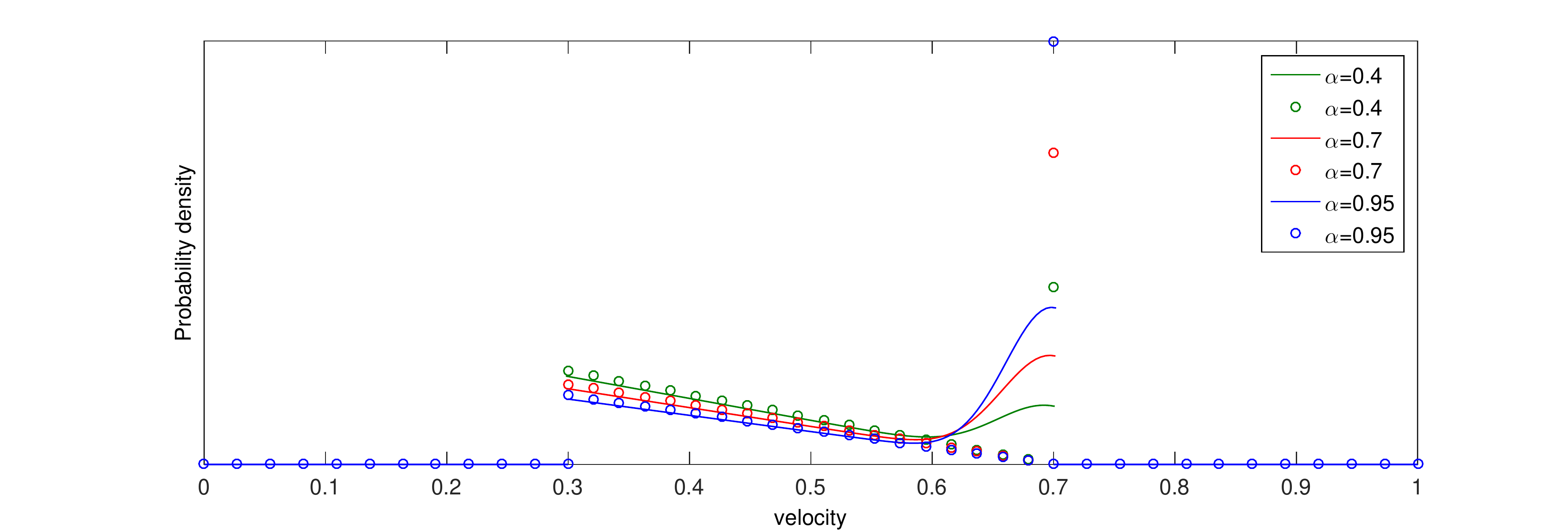}
	\caption{Comparison between discrete (from \cite{[BDF12]}) and continuous transition probability densities $\mathcal{A}$. Left: interaction with faster particles: $0.3=v_*< v^*=0.7$, $u=1$ and ${{\rho_p}}=0.6$. Right: interaction with slower particles: $0.7=v_*> v^*=0.3$, $u=1$ and ${{\rho_p}}=0.6$.}
	\label{Fig1}
\end{figure}

Now the global transition probability density $\mathcal{A}[f]=\mathcal{A}[f](v_* \to v,v^*,u)$ can be rewritten accordingly:
\begin{eqnarray}
\mathcal{A}[f]
&=&
\bigg[H(v^*-v_*)\Big((1-P)\frac{e^{-\frac{|v-v_*|^2}{\sigma_1}}}{Z(\sigma_1)}+ P\frac{(v_M-v)^2}{Z_2}\Big)\nonumber \\
\label{Table2}
&&
\hspace{-2 cm}+(1-H(v^*-v_*))\Big(P\frac{e^{-\frac{|v-v_*|^2}{\sigma_2}}}{Z(\sigma_2)}+ (1-P)\frac{v_M-v}{Z_1}\Big)\bigg]\chi_{[v_m,v_M]}(v).
\end{eqnarray}
We use the shorthand notation $P=\alpha u(1-{{\rho_p}})$, being $H(\cdot)$  is the Heaviside  function and $\chi_B$ the characteristic function of a set $B$.

\subsection{Modeling external action}

Let us now consider the modeling of external action, which enforce a prescribed speed, as it occurs for example in the presence of  tollgates.
The structure of this term is reported in Eq.~(\ref{T}), where  $f_e$ is a given function (for instance a step-wise function) of the prescribed velocity $v_e= v_e(x)$. Therefore, the simplicity of Eq.~(\ref{T}) entails that only the rate $\mu$ needs to be modeled. Following the same rationale applied to the encounter rate $\eta$ in Section \ref{ss:enc}, the following expression can be used:
\begin{equation}
{\mu[f,x]}= \mu_0(x) \, (1 + \gamma_\mu \, {{\rho_p}}[\rho]),
\label{mu}
\end{equation}
where $\gamma_\mu$ is the growth coefficient and ${{\rho_p}}$ is the perceived density, whereas $\mu_0$ accounts for the spatial localization of the external actions.

\subsection{Parameters and final model}
In this section we have proposed a simple modeling of the interaction terms to be implemented into the structure provided by Eq.~(\ref{structure}). The model includes the following parameters, which relate to specific different phenomena in vehicular traffic flows:
\begin{itemize}
	\item \vskip.2cm \noindent $\a$ models the quality of the road-weather conditions;
	\item \vskip.2cm \noindent $\ell_v$ is the length of the sensitivity zone $\Omega_\ell$. It depends on the quality of the road-weather conditions $\a$ by means of $\ell_v = \a L$;
	
	\item\vskip.2cm \noindent $\eta_0$ is the minimal value of the encounter rate; 
	
	\item \vskip.2cm \noindent $\gamma_\eta$ and $\gamma_\mu$ which are, respectively, the growth coefficients of the encounter rate $\eta$ and the intensity of the action $\mu$.
\end{itemize}
Bearing in mind the proposed modeling of the interaction terms, we get our definitive traffic model:
\begin{eqnarray}
	\label{DerivedModel}
	&&\hspace{-1.5 cm} \partial_t f(t,x,v,u) + v \partial_x f(t,x,v,u) + \partial_v \Big( \mathcal{F}[f](t,x,v,u)\partial_vf(t,x,v,u)\Big)  \\
	&=&  \int_{[0,1]^3}{\eta[f]}\mathcal{A}[f](v_*\to  v)\,f(t,x,v_*,u)\, f(t, x, v^*, u^*)dv_*\,dv^*\, du^*\nonumber \\
	&&  -f(t, x, v,u)\int_{[0,1]^2} {\eta[f]}\, f(t, x, v^*, u^*)dv^*\, du^*
	+{\mu[f,x]}\Big(f_e(x, v_e(x)) - f(t, x, v,u)\Big).\nonumber
	\end{eqnarray}

It is worth mentioning that in this paper we have proposed a probability density~(\ref{Table2}) by drawing inspiration on the table of games proposed by~\cite{[BDF12]}, that features {discrete velocity and activity variables}. Moreover, we {included} nonlinear interactions by taking into account the perceived density ${{\rho_p}}$ and mean field interactions, which are considered as one of the paradigms of complexity.  A few works proposing to model the probability density in terms of a continuous velocity variable can be found in the vehicular traffic literature. Namely, the recent work by~\cite{[PSTV17]}. In this paper, the authors proposed the quantified acceleration referred as $\delta$ model, in which the post-interactions after an acceleration is obtained by a velocity jump. The second model is based on the paper by~\cite{[KW97]} which assumed that the interaction result between the candidate and field particles is uniformly distributed in a velocity interval.

\section{Constructive iterative method of solutions} 
\label{sec4}
In this section we propose a scheme to construct a solution of  (\ref{DerivedModel}) by means of an iterative method. We shall pose the (linear) approximating problems on the whole space $[0,\infty]\times{\mathbb R}^3$ in order to rely on well-known results for linear transport equations and then we will prove that this amounts to a periodic extension of the linearization of the original problem.  
minimum of mathematical requirements. 
Thus, we will first consider a linearized version of model (\ref{DerivedModel}) on $[0,\infty]\times{\mathbb R}^3$ and state some properties that will enable us to construct approximate solutions.

\subsection{Linearized transport model} 

As our starting point we consider the following  transport model on $(t,x,v,u)\in [0,\infty]\times{\mathbb R}^3 $:
\begin{equation}
\label{uno}
\partial_t f+ v \partial_x f + \partial_v \big( \mathcal{F}[h]f\big) = g ,\qquad f(0,x,v,u)=f_0(x,v,u)\:.
\end{equation}
Here $h$ and $g$ are given; these are smooth functions on $[0,\infty]\times{\mathbb R}^3$. We assume those to be 1-periodic on $x$ and also to vanish for $(v,u)\notin [0,1]^2$. The initial data  $f_0$ is assumed to be smooth, 1-periodic on $x$ and compactly supported on $(v,u) \in [0,1]^2 $. The acceleration term $\mathcal{F}[h]$ is given by (\ref{acceleration2}) for $(t,x,v,u)\in [0,\infty]\times\RR^3$ and in this fashion it is 1-periodic on $x$.  
Note that our original model (\ref{DerivedModel}) is posed on $(t,x,v,u)\in [0,\infty]\times[0,1]^3$; we start by proving that, under natural hypothesis, solutions of (\ref{uno}) preserve the periodicity property on $x$ and the support property for $v$ and $u$.

\begin{Lemma}[Periodicity and support]
\label{lemma1}
Let $h$, $g$ and $f_0$ be smooth functions, vanishing for $(v,u)\notin [0,1]^2$ and 1-periodic on $x$. Then, the function $f$ verifying 
 (\ref{uno}) does also vanish  for $(v,u) \notin [0,1]^2$ and is 1-periodic on $x$, for all $t\geq 0$.
 \end{Lemma}

\begin{proof} Thanks to well-known results, the solution of (\ref{uno})  exists, is unique and can be represented in terms of the characteristic curves. Denote these curves by $(X(s),V(s),U(s))$, they depend on an initial tuple $(t,x,v,u)$ and are determined as the solutions to
\begin{eqnarray}
&&X'(s)=V(s), \ V'(s)={\cal F}[h](s,X(s),V(s),U(s)), \ U'(s)=0, \quad 0\leq s\leq t , \nonumber \\ 
&&X(t)=x, \ V(t)=v, \ U(t)=u. \label{pvichar}
\end{eqnarray}
We will also use the notation $(X(s,x),V(s,v),U(s,u))$ whenever we need to stress the dependence on initial conditions. We remark that under our assumptions the transport field can be easily shown to be locally Lipschitz; this ensures that the associated charateristic curves given by (\ref{pvichar}) are well-defined objects in the classical sense. 
 
Let us show that the following equality
\begin{equation}
\Big(X(s;x)+k,V(s),u\Big)=\Big(X(s;x+k),V(s),u\Big), \ \forall k\in {\mathbb N},
\label{perio}
\end{equation}
holds. Actually, we only have to observe that the acceleration term $\mathcal{F}[h]$ is 1-periodic w.r.t. $x$ provided $h$ is. Then we note that the initial value problems verified by both sides of  equality (\ref{perio}) are the same. The uniquenes of this IVP concludes the proof of (\ref{perio}).  Note in passing that, denoting by $J(s,t)$ the Jacobian of the change $(x,v,u)\mapsto (X(s),V(s),U(s))$ (which  depends implicitly on $x$), we can deduce the same periodicity property  (\ref{perio})  for $J(s,t)$; for that we just take derivatives on (\ref{perio}) to obtain the jacobian and use the smoothness of $h$.

Next we prove the following: if $V(t)=v \notin [0,1]$ then the $V$-component of the solution of (\ref{pvichar}) verifies that $V(s)\notin [0,1]$. 
It is equivalent to show the same property for the solution of the ``forward'' Cauchy problem (\ref{pvichar}), that is, we take 
 $t=0$ and $s\geq 0$ and we show that $V(s;v) \in [0,1]$ when $V(0)=v\in [0,1]$, for all $s\geq 0$. We argue by contradiction.
Recall 
 that 
\[
\mathcal{F}[h] =  \alpha \, u\, \ell_v \,\Big (\mathcal{F}_1[h](t,x) - v\ \mathcal{F}_2[h](t,x)\Big)
 \]
and note that  $\mathcal{F}_2[h]\geq \mathcal{F}_1[h] $. Assume that, for some positive time $s>0$, we have $V(s;v)>1 $, and then define   $ s_1=\inf \{ s\geq 0 : V(s;v)>1\}$. Thus, $V(s_1;v)=1$ and there exists some $\epsilon>0$ such that $V(s_1+\epsilon;v)>1$ and $V'(s_1+\epsilon;v)>0$. We arrive to a contradiction because
\begin{eqnarray*}
0&<&V'(s_1+\epsilon;v)= \alpha \, u\, \ell_v \,\Big (\mathcal{F}_1[h](t,X(s)) - V(s_1+\epsilon;v) \mathcal{F}_2[h](t,X(s))\Big) \\
&\leq& 
 \alpha \, u\, \ell_v \,\Big (\mathcal{F}_1[h](t,X(s)) - \mathcal{F}_2[h](t,X(s))\Big)\leq 0.
\end{eqnarray*}
Analogously, arguing by contradiction we can prove that, for any positive time $s>0$, we have $V(s;v)\geq 0$, which concludes the argument. To prove the analogous property with respect to the variable $u$ is straightforward, given that $U(s)=u$ for all $s$.

Finally, we conclude the proof of Lemma \ref{lemma1} by writting the solution of (\ref{uno})  as
\begin{equation}
\label{dos}
f(t,x,v,u)=J(0,t)f_0(X(0),V(0),u)+\int_0^t J(s,t)g(s,X(s),V(s),u) ds,
\end{equation}
and using the properties of $X$, $V$ and $U$ we just proved above. 
\end{proof}

Standard transport theory results allows us to prove the next Lemma, where $\| \cdot \|_{L^p} $ stands for the norm on $ L^p([0,1]^3)$.
\begin{Lemma}[$L^p$ estimates]
\label{lemma2}
Given   $h$, $g$ and $f_0$ verifying the hypothesis of Lemma \ref{lemma1} and such that 
\[
\|f_0\|_{L^p} + \|h(t,\cdot)\|_{L^p} + \|g(t,\cdot)\|_{L^p} <C<\infty
\]
for all $t\geq 0$ and $p=1,\infty$, then, (\ref{uno}) has a unique solution given by (\ref{dos}) and verifying 
\[
 \|f(t,\cdot)\|_{L^p} <C(T)<\infty,
\]
for any $T>0$ and  $t\in [0,T]$.  \end{Lemma}
 
\begin{proof}
 The theory of linear transport equations shows that the Jacobian satisfies $\partial_s J(s,t) = J(s,t) \partial_v(\mathcal{F}[h])(s,X(s),V(s),u)$ and $J(t,t)=1$. Then, using (\ref{accelerationII}), the equation for $J$ reads 
 $\partial_s J(s,t) = - \alpha\, J(s,t)  \, u \, \ell_v \, \mathcal{F}_2[h](s,X(s))$. Upon integration, 
 \[
 J(s,t)= \exp\Big\{-\int_t^s \alpha  \, u \, \ell_v \, \mathcal{F}_2[h](\tau,X(\tau))  d \tau\Big\}.
\]
We remark that $J$ is not compactly supported in $(v,u)$; However, all the instances of $J$ in (\ref{dos}) are multiplied by  functions that are compactly supported in $(v,u)$. Then, taking $L^p$ norms on (\ref{dos}) we obtain the desired estimate with  
 \[
 C(T)= e^{T  \|h\|_{\infty} } \|f_0 \|_p+ T e^{T  \|h\|_{\infty} }   \|g\|_{p}\:.
\]
When $p=1$ this can be refined to 
\[
C(T)=\|f_0\|_1 + \int_0^t \|g(s)\|_1 ds.
\]
 \end{proof}

 \subsection{Existence for the full traffic model} 
 Let us now address the full model (\ref{DerivedModel}), posed on the whole space $[0,\infty]\times{\mathbb R}^3$ and with a  perceived density ${{\rho_p}}$ given by the Lipschitz function introduced on on section \ref{sec31}. In order to do that,  let us consider, for $n\in {\mathbb N}$, the following iterative scheme:
  \begin{equation}
	\label{DerivedModeln}
	\partial_t f^{n+1}+ v \partial_x f^{n+1} + \partial_v ( \mathcal{F}[f^n] f^{n+1}) = g^n.
	\end{equation} 
This is the linear transport eq.  (\ref{uno}) with $h=h^n=f^n$, where $g=g^n$ is given by:
 \begin{eqnarray*}	
 g^n &:=& \int_{[0,1]^3}\eta[f^n]\mathcal{A}[f^n](v_*\to  v)\,f^{n}(t,x,v_*,u)\, f^{n}(t, x, v^*, u^*)dv_*\,dv^*\, du^*\nonumber 
 \\
&&  -f^n \int_{[0,1]^2} \eta[f^n]\, f^n(t, x, v^*, u^*)dv^*\, du^*  +\mu_n[f^n,x]\Big(f^n_e(x, v_e(x)) - f^n \Big).\nonumber
\end{eqnarray*}
We prescribe smooth initial data $f^{n+1}(0,x,v,u)=f_0^{n+1}(x,u,v)$ that are 1-periodic with respect to $x$ and whose $(v,u)$-support is contained on $[0,1]^2$. For $n=0$, we can take for example  $f^0(t,x,v,u)=f_0^{0}(x,v,u)$. 
Here we also assume that the equilibrium $f^n_e(x, v_e(x))$ is  
smooth, 1-periodic with respect to $x$ and such that $v_e(x)\in[0,1] $ for all $x$ and that 
$$\int_{[0,1]} f^n_e(x, v_e(x)) dx= \int_{[0,1]^3} f^n_0(x, v,u) dx \, dv \, du. $$ 
Moreover, we take $\mu_n[f,x]:=\mu_0^n(x) (1+\gamma_\mu \rho_p(\rho))$ with $\mu_0^n$ being smooth and 1-periodic with respect to $x$.

We impose the following convergences for the given initial datum, equilibrium density and external actions intensity function:
\begin{equation}
\label{convf0}
\begin{array}{l}
f_0^{n} \stackrel{n\to \infty}{\longrightarrow} f_0 \mbox{ in } L^1([0,1]^3), 
\quad  \|f_0^{n} \|_{L^\infty([0,1]^3)} \leq C< \infty, \  \forall n\in \NN,\\[5 pt]
f_e^{n} \stackrel{n\to \infty}{\longrightarrow} f_e \mbox{ in } L^1(0,1), 
\quad  \|f_e^{n} \|_{L^\infty(0,1)} \leq C< \infty, \  \forall n\in \NN,\\[5 pt]
\mu_0^{n} \stackrel{n\to \infty}{\longrightarrow} \mu_0 \mbox{ in } L^1(0,1), 
\quad  \|\mu_0^{n} \|_{L^\infty(0,1)} \leq C< \infty, \  \forall n\in \NN.
\end{array}
\end{equation}
The next result gives sufficient conditions to ensure that the sequence $f^n$ converges to a solution of (\ref{DerivedModel}) with $f_0$ as initial data. 
\begin{Theorem}
Let $f_0^n$,$f^n_e$ and $\mu_0^n$ smooth sequences of initial datum, equilibrium densities and external actions intensity function respectively, all of them being 1-periodic with respect to $x$ and verifying (\ref{convf0}). Let the sequence $f_0^n$ be supported on $[0,1]^2$ with respect to $(v,u)$.
Then, by assuming the $L^\infty ([0,\infty]; L^1( [0,1]^3))$-convergence of the sequence $f^n$ of solutions of (\ref{DerivedModeln}), we can ensure that the limit $f$ solves (\ref{DerivedModel}) in a weak sense.
\end{Theorem}
\begin{proof} 
We can prove by induction that $h^n$ and $g^n$ verify the hypothesis described on Lemma \ref{lemma1} and \ref{lemma2}. Therefore, the sequence $f^n$ is 1-periodic with respect to $x$ and supported on $[0,1]^2$ with respect to $(v,u)$. In order prove that its limit is a solution of  (\ref{DerivedModel}), we start from the weak formulation of (\ref{DerivedModeln}); let a test function $\Phi  \in C_0^\infty ([0,\infty)\times (0,1)^3)$, then we have
\begin{eqnarray*}
0&=&\int_{[0,1]^3} f^{n+1}_0\Phi  (0,x,v,u) dxdvdu + \int_0^\infty\!\!\!\!\int_{[0,1]^3} f^{n+1} 
\Big( \partial_t \Phi  +v\partial_x \Phi  \Big) dt dxdvdu \\
&&+ \int_0^\infty\!\!\!\!\int_{[0,1]^3}  f^{n+1}   \mathcal{F}[f^n] \partial_v \Phi \, dt dxdvdu  
+ \int_0^\infty\!\!\!\!\int_{[0,1]^3} \mu[f^n,x]\Big(f^n_e- f^n \Big) \Phi \, dt dxdvdu    \\
&&+ \int_0^\infty\!\!\!\!\int_{[0,1]^6} \eta[f^n]\mathcal{A}[f^n] \,f^{n}(t,x,v_*,u)\, f^{n}(t, x, v^*, u^*)  \Phi \, dv_*\,dv^*\, du^*  dt dxdvdu\nonumber 
 \\
&&  -  \int_0^\infty\!\!\!\!\int_{[0,1]^5}  \eta[f^n]\,  f^n(t,x,v,u) f^n(t, x, v^*, u^*) \Phi \, dv^*\, du^*  dt dxdvdu.
\end{eqnarray*}
Using the convergence hypothesis for $f^n$ and the boundness of $\vf$, $\mathcal{A}[f^n]$, $\eta[f^n]$ and $\mu_n[f^n,x]$, it is easy to deduce that each term above converges. To illustrate it, let us show for example the convergence of the third term, which reads
\[
\int_0^\infty\!\!\!\!\int_{[0,1]^3}  f^{n+1}   \mathcal{F}[f^n] \partial_v \Phi \, dt dxdvdu 
\]
\[
 =
\int_0^\infty\!\!\!\!\int_{[0,1]^3}\! \int_{\Lambda}  \partial_v \Phi \vf  f^{n+1} (t,x,v,u)   f^n(t,x^*,v^*,u^*)\,dx^* dv^* du^* dt dxdvdu.
\]
We use now standard results for transport equations, e.g. the results in \cite{[BCNS]}, to deduce the convergence of the product measure,
\[
f^{n+1} (t,x,v,u)   f^n(t,x^*,v^*,u^*) \to f(t,x,v,u)   f(t,x^*,v^*,u^*) \ in \ C(0,T;{\cal M}([0,1]^6)-weak-*),
\]
Actually we can show that in this case the convergence takes place in $C(0,T;L^1([0,1]^6))$. This allows us to pass to the limit, thus obtaining
\[
\int_0^\infty\!\!\!\!\int_{[0,1]^3}  f^{n+1}   \mathcal{F}[f^n] \partial_v \Phi\,  dt dxdvdu   \to 
\int_0^\infty\!\!\!\!\int_{[0,1]^3}  f   \mathcal{F}[f] \partial_v \Phi \, dt dxdvdu.
\]
The rest of the terms are handled in a similar way. Note here that, thanks to the boundness of the perceived density $\rho_p[\rho]$ and (\ref{convf0}), the three $\mathcal{A}[f^n]$, $\eta[f^n]$ and $\mu_n[f^n,x]$ converge pointwise and in $L^\infty([0,\infty]\times [0,1]^3)-weak^*$ to 
$\mathcal{A}[f]$, $\eta[f]$ and $\mu[f,x]$ respectively. Thus, we conclude the convergence of all the terms and then,
\begin{eqnarray*}
0&=&\int_{[0,1]^3} f_0\Phi  (0,x,v,u) dxdvdu + \int_0^\infty\!\!\!\!\int_{[0,1]^3} f 
\Big( \partial_t \Phi  +v\partial_x \Phi  \Big) dt dxdvdu \\
&&+ \int_0^\infty\!\!\!\!\int_{[0,1]^3}  f   \mathcal{F}[f] \partial_v \Phi \, dt dxdvdu  
+ \int_0^\infty\!\!\!\!\int_{[0,1]^3} \mu[f,x]\Big(f_e- f \Big) \Phi \, dt dxdvdu    \\
&&+ \int_0^\infty\!\!\!\!\int_{[0,1]^6} \eta[f]\mathcal{A}[f] \,f(t,x,v_*,u)\, f(t, x, v^*, u^*)  \Phi \, dv_*\,dv^*\, du^*  dt dxdvdu\nonumber 
 \\
&&  -  \int_0^\infty\!\!\!\!\int_{[0,1]^5}  \eta[f]\,  f(t,x,v,u) f(t, x, v^*, u^*) \Phi \, dv^*\, du^*  dt dxdvdu,
\end{eqnarray*}
that is, $f$ is a weak solution of (\ref{DerivedModel}) with initial datum $f_0$.
\end{proof}


\medskip \noindent {\it Acknowledgements.} 
J.C. and J.N. are partially supported by Junta de Andaluc\'ia Project P12-FQM-954 and MINECO Project~RTI2018-098850-B-I00. J.C. is supported by Universidad de Granada (``Plan propio de investigaci\'on, programa 9'') through FEDER funds. M.Z. was supported by CNRST (Morocco), project ``Mod\`eles Math\'ematiques appliqu\'es \`a l’environnement, \`a l’imagerie m\'edicale et aux biosyst\`emes''.


\end{document}